\documentclass[12pt]{amsart}

\usepackage[T1]{fontenc}
\usepackage[utf8]{inputenc}
\usepackage[protrusion=true,expansion=false]{microtype}
\usepackage{amsmath,amssymb,aliascnt,needspace,color, comment}
\usepackage[hidelinks]{hyperref}
\usepackage[nameinlink,capitalise,noabbrev]{cleveref}
\usepackage[margin=1in]{geometry}

\allowdisplaybreaks
\numberwithin{equation}{section}

\newtheorem{theorem}{Theorem}[section]
\newaliascnt{lemma}{theorem}
\newtheorem{lemma}[lemma]{Lemma}
\aliascntresetthe{lemma}
\newaliascnt{claim}{theorem}
\newtheorem{claim}[claim]{Claim}
\aliascntresetthe{claim}
\newaliascnt{proposition}{theorem}
\newtheorem{proposition}[proposition]{Proposition}
\aliascntresetthe{proposition}
\newaliascnt{corollary}{theorem}

\aliascntresetthe{corollary}
\newaliascnt{problem}{theorem}

\aliascntresetthe{problem}
\theoremstyle{definition}
\newaliascnt{definition}{theorem}

\aliascntresetthe{definition}
\theoremstyle{remark}
\newaliascnt{remark}{theorem}

\aliascntresetthe{remark}

\crefname{theorem}{Theorem}{Theorems}
\Crefname{theorem}{Theorem}{Theorems}
\crefname{lemma}{Lemma}{Lemmas}
\Crefname{lemma}{Lemma}{Lemmas}
\crefname{proposition}{Proposition}{Propositions}
\Crefname{proposition}{Proposition}{Propositions}
\crefname{corollary}{Corollary}{Corollaries}
\Crefname{corollary}{Corollary}{Corollaries}
\crefname{definition}{Definition}{Definitions}
\Crefname{definition}{Definition}{Definitions}
\crefname{remark}{Remark}{Remarks}
\Crefname{remark}{Remark}{Remarks}

\newcommand{\x}{\times}
\newcommand{\omegaone}{{\omega_1}}

\newcommand{\Pow}{\operatorname{P}}

\newcommand{\HFCw}{\mathrm{HFC}_\mathrm{w}}
\newcommand{\set}[1]{\left\{#1\right\}}

\newcommand{\Un}{\bigcup}

\newcommand{\sset}[2]{\{\, #1 : #2\, \}}
\newcommand{\fd}{\mathfrak{d}}

\newcommand{\card}[1]{\left| #1 \right|}

\newcommand{\sub}{\subseteq}

\newcommand{\fc}{\mathfrak{c}}

\newcommand{\ww}{\omega^\omega}

\makeatletter
\newenvironment{citelist}{%
  \def\@cite##1##2{{\@citestyle\citeform{##1}\if@tempswa, ##2\fi}}%
  \@citestyle[\ignorespaces
}{\unskip]}
\makeatother

\title[The D-space Problem]{A consistent solution of the D-space Problem}
\author{Boaz Tsaban}
\address{Department of Mathematics, Bar-Ilan University}
\email{tsaban@math.biu.ac.il}

\begin{document}

\begin{abstract}
Settling a central problem in set-theoretic topology, and many related problems, we prove that there is, consistently, a regular Lindel\"of space of cardinality $\aleph_1$ that is not a D-space.
The existence of such a space is independent of ZFC; we establish its equivalence to a weak form of the Continuum Hypothesis.
The proof of the main theorem is adjustable to produce consistent examples with a variety of stronger properties.
\end{abstract}

\maketitle

\bigskip

\section{Introduction}\label{sec:introduction}

We consider an old, basic question about covering properties.
By \emph{space} we mean a Hausdorff topological space.
Recall that a space is \emph{Lindel\"of} if every open cover of the space has a
countable subcover.
A space $X$ is a \emph{D-space} if every open cover $\sset{N(x)}{x\in X}$
with $x\in N(x)$ for all $x\in X$ (a \emph{neighborhood assignment}) has a subcover
$\sset{N(x)}{x\in D}$ where the set $D$ is closed and discrete.
The \emph{D-Space Problem}~\cite{vanDouwenPfeffer1979} asks whether every regular
Lindel\"of space is a D-space.
This problem is listed in Hru\v{s}\'ak and Moore's compilation of twenty central problems in set-theoretic
topology~\cite[Problem~14]{HrusakMoore2007}.
It is sometimes stated for regular \emph{hereditarily} Lindel\"of spaces~\cite[Question~1]{Eisworth2007}.

There are excellent surveys of the D-space problem and its variations, the progress made
while addressing them, and the inherent
difficulties~\cite{Eisworth2007, GruenhageSurvey2011}.
Some later developments are cited in the bibliography;
in particular,
Aurichi~\cite[Corollary~2.7]{Aurichi2010} proved that every \emph{Menger} space (a
property stronger than Lindel\"of) is D, and Soukup and Szeptycki proved that Jensen's $\diamondsuit$ principle (a consistent
hypothesis that we will consider later) implies that there is a \emph{Hausdorff} hereditarily Lindel\"of space that is not
D~\cite[Corollary~3.8]{SoukupSzeptycki2012}, and a regular hereditarily Lindel\"of space that is not
\emph{strongly} D~\cite[Theorem~1]{SoukupSzeptycki2019}.
Jensen's $\diamondsuit$ is stronger than the Continuum Hypothesis (CH).
We establish the following result.

\begin{theorem}\label{thm:main}
Assuming CH, there is a regular, hereditarily Lindel\"of space of cardinality $\aleph_1$ that is not
a D-space.
\end{theorem}

This answers the question whether all regular hereditarily Lindel\"of spaces are provably D\@.
It remains open whether it is \emph{consistent} that every regular hereditarily Lindel\"of space is D\@.

The existence assertion in \Cref{thm:main} cannot be established in ZFC\@:
Let $\fd$ (the \emph{dominating number}) be the minimal cardinality of a set of functions in $\ww$
such that every  function in $\ww$ is eventually dominated by some function from this set.
While $\aleph_1\le\fd$ and CH implies equality, it is consistent that $\aleph_1<\fd$~\cite{Blass2010}.
By a classic result of Hurewicz, every Lindel\"of space of cardinality smaller than $\fd$ is Menger (and thus D).
Moreover, in \Cref{thm:d} we establish that the assertion $\fd=\aleph_1$ is sufficient for, and thus \emph{equivalent} to, the existence
of a space as in \Cref{thm:main}.

The proof method is adjustable, and we provide examples with additional restrictive properties.
Assuming CH, there is an example $X$ such that the infinite power $X^\omega$ (and thus also every finite power)
is hereditarily Lindel\"of (\Cref{thm:CH-hL-power}).
Assuming Jensen's $\diamondsuit$, there is an example with even stronger properties (\Cref{thm:strong-HFCw}).
These results settle a large number of open problems, which we survey in \Cref{sec:conclusion}.

\subsection*{Acknowledgments}
The works of Soukup and Szeptycki~\cite{SoukupSzeptycki2012,SoukupSzeptycki2019} were a major inspiration for this paper.
The present paper grew out of discussions with ChatGPT about my old (2011, perhaps earlier) informal notes on
the D-space problem and related questions.
It took hundreds of hours, thousands of prompts, through several consecutive ChatGPT versions, where both of us made substantial contributions.
I thank my wife for her support and encouragement during this intensive period.
I hope she survives the forthcoming revision periods, too; I also hope that I do.

\section{Auxiliary results}\label{sec:auxiliary}\label{sec:reduction}

Identify the set $\Pow(\omegaone)$ with the Cantor cube $2^\omegaone$ by characteristic functions.
For an ordinal $\alpha<\omegaone$, the subspace
$\Pow(\alpha):=\sset{x\sub\omegaone}{x\sub\alpha}$ is closed and
homeomorphic to the second countable cube $2^\alpha$.
All convergence assertions in this paper are in the Cantor cube.

For a finite set $s=\set{s_0<\cdots<s_{m-1}}\sub\omegaone$ and a subset
$b\sub m$, let $s[b]:=\sset{s_i}{i\in b}$.
Let $[b,s]:=\sset{x\sub\omegaone}{x\cap s=s[b]}$ and, for a family
$F\sub[\omegaone]^m$, let
\[
[b,F]:=\Un_{s\in F}[b,s].
\]
The sets $[b,s]$ are basic clopen sets, and the set $[b,F]$ is open.
When the family $F$ is infinite and its elements are pairwise disjoint,
the set $[b,F]$ is also dense.
If, in addition, the set $F\cap[\alpha]^m$ is infinite, then the set $[b,F]\cap \Pow(\alpha)$
is open dense in the subspace $\Pow(\alpha)$.

\begin{lemma}[{Juh\'asz~\cite[1.1(ii)]{Juhasz2002}}]\label{lem:HL}
A subspace $X\sub\Pow(\omegaone)$ is hereditarily Lindel\"of if and only if,
for each $0<m<\omega$, each subset $b\sub m$, and each pairwise
disjoint family $F\sub[\omegaone]^m$, there is a countable family
$F'\sub F$ such that $X\cap[b,F]\sub[b,F']$.
\end{lemma}

In the following lemma, we do not assume that the families $F_n$, for $n<\omega$, are distinct.

\begin{lemma}\label{lem:selection}
Suppose that $I\sub\alpha<\omegaone$, and for each $n<\omega$ we have
$m_n<\omega$, the family $F_n\sub[\alpha]^{m_n}$ is infinite, its
elements are pairwise disjoint, and infinitely many of them are disjoint
from the set $I$.
There is a set $x\sub\alpha$ containing $I$ such that for each $n<\omega$
and each subset $b\sub m_n$, we have $x\in [b,s]$ for infinitely many $s\in F_n$.
\end{lemma}

\begin{proof}
Enumerate
\[
\sset{(n,b)}{n<\omega,\ b\sub m_n}=\sset{(n_k,b_k)}{k<\omega}
\]
such that each pair repeats infinitely often.
By recursion, for each $k<\omega$ pick a set $s_k\in F_{n_k}$ disjoint
from the set $I$ and from the finite set $\Un_{j<k}s_j$.
Let $x:=I\cup\Un_{k<\omega}s_k[b_k]$.
Then $x\in [b_k, s_k]$ for all $k<\omega$.
For each $n<\omega$ and each subset $b\sub m_n$, we have
$(n_k,b_k)=(n,b)$,
$x\in [b_k, s_k]=[b,s_k]$, and
$s_k\in F_{n_k}=F_n$ for infinitely many $k<\omega$.
\end{proof}

\begin{proposition}\label{prp:approximation}
Suppose that $0<m<\omega$, $F\sub[\omegaone]^m$, and
$x_\alpha\in[\omegaone]^{\le\omega}$ for $\alpha<\omegaone$.
For each ordinal $\alpha<\omegaone$, there is an ordinal
$\alpha<\delta<\omegaone$ such that, for each set
$t\in F\setminus[\delta]^m$, there are distinct sets
$s_0,s_1,\dotsc\in F\cap[\delta]^m$ such that
$x_{s_{n,i}}\longrightarrow x_{t_i}\cap\delta$ for all $i<m$.
\end{proposition}

\begin{proof}
Fix $\alpha<\omegaone$ and let $\alpha_0:=\alpha$.
Given an ordinal $\alpha_n<\omegaone$ and a finite subset $a\sub\alpha_n$,
the set
\[
P:=\sset{(x_{t_0}\cap a,\dotsc,x_{t_{m-1}}\cap a)}{t\in F\setminus[\alpha_n]^m}
\]
is finite.
Pick a finite set of representatives $R_a\sub F\setminus[\alpha_n]^m$ such that
for each element $(x_{t_0}\cap a,\dotsc,x_{t_{m-1}}\cap a)\in P$ there is an element $s\in R_a$
with
\begin{equation}
\label{rep}
(x_{s_0}\cap a,\dotsc,x_{s_{m-1}}\cap a)=(x_{t_0}\cap a,\dotsc,x_{t_{m-1}}\cap a).
\end{equation}
Let $F_n:=\Un_{a\in[\alpha_n]^{<\omega}}R_a\sub F\setminus[\alpha_n]^m$.
Then the family $F_n$ is countable, and for each set
$t\in F\setminus[\alpha_n]^m$ and each finite set $a\sub\alpha_n$,
some $s\in F_n$ satisfies \Cref{rep}.

Choose an ordinal $\alpha_n<\alpha_{n+1}<\omegaone$ with
\[
\Un_{s\in F_n} ( s\cup x_{s_0}\cup\dotsb\cup x_{s_{m-1}})\sub\alpha_{n+1}.
\]
Let $\delta:=\sup_{n<\omega}\alpha_n$.

Choose increasing finite sets $a_0\sub a_1\sub \dotsb$ such that $a_n\sub\alpha_n$ for all $n<\omega$ and
$\Un_{n<\omega}a_n=\delta$.
Let $t\in F\setminus[\delta]^m$.
For each $n<\omega$, there is a set $s_n\in F_n$ such that \Cref{rep} holds for $s:=s_n$ and $a:=a_n$.
Since $\max s_n\in[\alpha_n,\alpha_{n+1})$ for all $n<\omega$, the sets $s_n$ are distinct.

Let $i<m$.
We have $x_{s_{n,i}}\sub\delta$.
For each coordinate $\xi\in x_{t_i}\cap\delta$, for all but finitely many $n$, we have $\xi\in a_n$ and thus
\[
 \xi \in x_{t_i}\cap a_n = x_{s_{n,i}}\cap a_n.
\]
It follows that $\xi\in x_{s_{n,i}}$ for all but finitely many $n$.
Similarly, for each coordinate $\xi<\delta$ with $\xi\notin x_{t_i}\cap\delta$, we have $\xi\notin x_{s_{n,i}}$ for all but finitely many $n$.
Thus, $x_{s_{n,i}}\longrightarrow x_{t_i}\cap\delta$ in $\Pow(\delta)$ and thus in $\Pow(\omegaone)$.
\end{proof}

The following proposition is not needed for the proof of the main theorem.
It is only used in the result about $\diamond$ that is announced later.

\begin{proposition}\label{prp:approximation-club}
Under the hypotheses of \Cref{prp:approximation}, the set $D$ of ordinals
$\delta<\omegaone$ satisfying its approximation conclusion is a club.
\end{proposition}

\begin{proof}
By \Cref{prp:approximation}, the set $D$ is unbounded.
It is also closed:
Let $\delta<\omegaone$ be a limit point of $D$, and $t\in F\setminus[\delta]^m$.
Let $\gamma<\omegaone$ be an ordinal with
$\delta\cup\Un_{\beta<\delta}x_\beta\sub\gamma$.
Choose increasing finite sets $a_0\sub a_1\sub \dotsb$ such that $\Un_{n<\omega}a_n=\gamma$.

Choose ordinals $\delta_0<\delta_1<\dotsb$ in $D$ with $\delta=\sup_{n<\omega}\delta_n$.
For each $n<\omega$, since $\delta_n\in D$ there are infinitely many sets $s\in F\cap[\delta_n]^m$ such that
\[
x_{s_i}\cap a_n = x_{t_i}\cap \delta_n\cap a_n
\]
for all $i<m$.
Pick such a set $s_n$ that is distinct from the sets $s_0,\dotsc,s_{n-1}$.
Then
\[
x_{s_{n,i}}\longrightarrow x_{t_i}\cap\delta
\]
for all $i<m$.
\end{proof}

In the following lemma, we do not assume that the enumeration of the
subspace is bijective.

\begin{lemma}\label{lem:nonD}
Suppose that the subspace $X:=\sset{x_\alpha}{\alpha<\omegaone}\sub\Pow(\omegaone)$
is Lindel\"of, $x_\alpha\sub\alpha$ for all $\alpha<\omegaone$, and for each
countable set $I\sub\omegaone$ such that the points $x_\alpha$, for $\alpha\in I$,
are distinct and form a closed discrete subspace of $X$, there is a point
in $X$ that contains the set $I$.
Then the subspace $X$ is not a D-space.
\end{lemma}

\begin{proof}
By choosing unique representatives, we have a set $J\sub\omegaone$
such that $X=\sset{x_\alpha}{\alpha\in J}$ and $x_\alpha\ne x_\beta$ for all
distinct indices $\alpha,\beta\in J$.
For each index $\alpha\in J$, since $x_\alpha\sub\alpha$, the set $N(x_\alpha):=\sset{x\in X}{\alpha\notin x}$
is a basic clopen neighborhood of the point $x_\alpha$.
Let $D=\sset{x_\alpha}{\alpha\in I}\sub X$, with $I\sub J$, be closed and discrete.
Since the space $X$ is Lindel\"of, the set $I$ is countable.
By the hypothesis, there is an index $\gamma\in J$ with $I\sub x_\gamma$.
Then for each index $\alpha\in I$, we have $x_\gamma\notin N(x_\alpha)$.
\end{proof}

\section{\texorpdfstring{Proof of \Cref{thm:main}}{Proof of the main theorem}}
\label{sec:CH}\label{sec:proof-main}

Assume the Continuum Hypothesis.
Enumerate
\begin{align*}
[\omegaone]^{\le\omega} &= \sset{q_\alpha}{\alpha<\omegaone},\\
[[\omegaone]^{<\omega}]^{\le\omega} &= \sset{F_\alpha}{\alpha<\omegaone},
\end{align*}
such that each element occurs cofinally often in the enumeration, and
$F_\alpha\sub[\alpha]^{<\omega}$ for all $\alpha<\omegaone$.
For the latter requirement, begin with an enumeration without this
requirement, and replace each family $F_\alpha$ with
$F_\alpha\cap[\alpha]^{<\omega}$.
Every countable family $F$ of finite subsets of $\omegaone$ is contained in $[\alpha]^{<\omega}$ for all large enough $\alpha<\omegaone$,
and thus still appears cofinally often in the enumeration.

We define points $x_\alpha\sub\alpha$ and open dense sets $O_\alpha\sub\Pow(\omegaone)$
by recursion on $\alpha<\omegaone$, allowing repetitions.
Each set $O_\alpha$ is either $\Pow(\omegaone)$ or has the form
$\bigcap_{b\sub m}[b,F_\alpha]$, where $0<m<\omega$ and the family
$F_\alpha\sub[\alpha]^m$ is infinite and its elements are pairwise disjoint.

At stage $\alpha$, we have the points $x_\beta$ and the open sets $O_\beta$ for $\beta<\alpha$.
Let
\[
G_\alpha:=\bigcap_{\beta<\alpha}O_\beta,
\]
where
$G_0:=\Pow(\omegaone)$.
We have $\emptyset\in G_0\cap\Pow(0)$.
If $q_\alpha\in \Un_{\beta\le \alpha}G_{\beta}\cap\Pow(\beta)$, let
\[
x_\alpha:=q_\alpha.
\]
Otherwise, set $x_\alpha:=\emptyset$.
We have $x_\alpha\sub\alpha$.

Suppose that for $0<m<\omega$ we have $F_\alpha\sub[\alpha]^m$, the elements of the family $F_\alpha$ are
pairwise disjoint, and there are distinct sets
$s_0,s_1,\dotsc\in F_\alpha$, and points
\[
p_0,\dotsc,p_{m-1}\in \Un_{\beta\le \alpha}G_{\beta}\cap\Pow(\beta)
\]
such that
\[
x_{s_{n,i}}\longrightarrow p_i
\]
for each $i<m$.
In this case, we call this stage \emph{active} and define
\[
O_\alpha:=\bigcap_{b\sub m}[b,F_\alpha].
\]
Otherwise, set $O_\alpha:=\Pow(\omegaone)$.
The set $O_\alpha$ is open and dense.
This completes the construction.

Let $X:=\sset{x_\alpha}{\alpha<\omegaone}$.
As a subspace of the Cantor cube, the space $X$ is zero-dimensional and thus regular.

\begin{claim}\label{cl:enumeration}
$X=\Un_{\alpha<\omegaone}G_\alpha\cap\Pow(\alpha)$.
\end{claim}

\begin{proof}
($\sub$) For each $\alpha<\omegaone$, either $x_\alpha=q_\alpha\in \Un_{\beta\le\alpha}G_\beta\cap\Pow(\beta)$,
or $x_\alpha=\emptyset\in G_0\cap\Pow(0)$.

($\supseteq$) Let $\beta<\omegaone$ and $x\in G_\beta\cap\Pow(\beta)$.
There are cofinally many ordinals $\alpha<\omegaone$ with $x=q_\alpha$.
Pick one that is greater than $\beta$.
Then $q_\alpha=x\in G_\beta\cap\Pow(\beta)$, and by the construction we have
$x_\alpha=q_\alpha=x$.
Thus $x\in X$.
\end{proof}

\begin{claim}\label{cl:truncations}
For each point $x\in X$ and each ordinal $\gamma<\omegaone$, we have
$x\cap\gamma\in\Un_{\beta\le\gamma}G_\beta\cap\Pow(\beta)$.
\end{claim}

\begin{proof}
Let $x\in X$ and $\gamma<\omegaone$.
By \cref{cl:enumeration}, there is an ordinal
$\alpha<\omegaone$ with $x\in G_\alpha\cap\Pow(\alpha)$.

If $\alpha\le\gamma$, then
$x\cap\gamma=x\in G_\alpha\cap\Pow(\alpha)$,
and we are done.

Thus, suppose that $\gamma<\alpha$.
For each ordinal $\beta<\gamma$, we have $x\in G_\alpha\sub O_\beta$.
If $O_\beta=\Pow(\omegaone)$ then $x\cap\gamma\in O_\beta$.
Suppose now that $O_\beta\ne\Pow(\omegaone)$.
Then stage $\beta$ is active, so for some $0<m<\omega$ we have
\[
O_\beta=\bigcap_{b\sub m}[b,F_\beta].
\]
Since every set $s\in F_\beta$ is a subset of $\beta<\gamma$,
we have $(x\cap\gamma)\cap s=x\cap s$, and thus
$x\cap\gamma\in O_\beta$ in this case, too.
It follows that $x\cap\gamma\in G_\gamma$, and thus
$x\cap\gamma\in G_\gamma\cap\Pow(\gamma)$.
\end{proof}

\begin{claim}\label{cl:HL}
The space $X$ is hereditarily Lindel\"of.
\end{claim}

\begin{proof}
We prove that \cref{lem:HL} applies.
Let $0<m<\omega$ and $b\sub m$, and let $F\sub[\omegaone]^m$ be an uncountable family whose elements are pairwise disjoint.
Let $\delta$ be the ordinal supplied by \cref{prp:approximation} and
$t\in F\setminus[\delta]^m$.
There are distinct sets $s_0,s_1,\dotsc\in F\cap[\delta]^m$ such that
\[
x_{s_{n,i}}\longrightarrow p_i:=x_{t_i}\cap\delta\in X
\]
for each $i<m$.

By \cref{cl:truncations}, for each ordinal $\delta\le\alpha<\omegaone$ we have
\[
p_i=p_i\cap\alpha\in\Un_{\beta\le\alpha}G_\beta\cap\Pow(\beta)
\]
for all $i<m$.
Let $B:=\sset{s_n}{n<\omega}$ and choose a stage $\alpha\ge\delta$
with $F_\alpha=B$.
Then this stage is active, and thus $O_\alpha\sub[b,B]$.

We have $X\setminus O_\alpha\sub\Pow(\alpha)$:
Let $x\in X\setminus O_\alpha$.
By \cref{cl:enumeration}, there is an ordinal $\beta<\omegaone$
with $x\in G_\beta\cap\Pow(\beta)$.
If $\alpha<\beta$, then $G_\beta\sub O_\alpha$ and thus $x\in O_\alpha$, a contradiction.
Thus $\beta\le\alpha$, and $x\sub\beta\sub\alpha$.

Since $\alpha<\omegaone$, the subspace $\Pow(\alpha)\cap[b,F]$ is second countable, and thus Lindel\"of.
It follows that its cover by the sets $[b,s]$, for $s\in F$, has a countable subcover.
Adding the corresponding members of $F$ to the set $B$ gives a countable family
$F'\sub F$ with $X\cap[b,F]\sub[b,F']$.
\end{proof}

\begin{claim}\label{cl:nonD}
The space $X$ is not a D-space.
\end{claim}

\begin{proof}
We verify the hypothesis of \cref{lem:nonD}.
Let $I\sub\omegaone$ be countable, and suppose that the points $x_\alpha$,
for $\alpha\in I$, are distinct and form a closed discrete set $D\sub X$.
Fix an active stage $\beta$.
At that stage there are a number $0<m<\omega$ and distinct sets
$s_0,s_1,\dotsc\in F_\beta$ such that for each $i<m$ the sequence $\langle x_{s_{n,i}} : n<\omega\rangle$
is convergent in $X$.
Since the ordinal $\beta$ is active, the members of the family $F_\beta$ are pairwise disjoint, and hence so
are the sets $s_0,s_1,\dotsc$.

Let $i<m$.
Suppose that the set $M:=\sset{n<\omega}{s_{n,i}\in I}$ is infinite.
Since the points $x_\alpha$, for $\alpha\in I$, are distinct,
the subsequence $\langle x_{s_{n,i}} : n\in M\rangle$
is a convergent sequence of distinct elements of $D$, a contradiction.
Thus, for all but finitely many $n$ we have $s_{n,i}\notin I$.
It follows that all but finitely many sets $s_n$ are disjoint from the index set $I$.
In particular, infinitely many members of the family $F_\beta$ are disjoint from $I$.

Choose an ordinal $\alpha<\omegaone$ with $I\sub\alpha$.
By \cref{lem:selection}, applied to the families $F_\beta$ at active stages $\beta<\alpha$,
there is a point $x\in G_\alpha\cap \Pow(\alpha)$ containing $I$.
If there are no active stages below $\alpha$, take $x:=I$.
In either case, we have $x\in G_\alpha\cap \Pow(\alpha)\sub X$.
By \cref{lem:nonD}, the space $X$ is not a D-space.
\end{proof}

Since every countable space is a D-space, the space $X$ is uncountable,
and thus $\card{X}=\aleph_1$.
This completes the proof of \cref{thm:main}.
\qed

\section{Results to be added in future revisions}

\makeatletter
\@namedef{r@cor:nonCH-hL-power}{{4.10}{}{}{}{}}
\@namedef{r@cor:nonCH-hL-power@cref}{{[corollary][10][4]4.10}{}{}{}{}}
\@namedef{r@sec:open-problem-consequences}{{7}{}{}{}{}}
\@namedef{r@sec:open-problem-consequences@cref}{{[section][7][]7}{}{}{}{}}
\@namedef{r@cor:covering-questions}{{7.1}{}{}{}{}}
\@namedef{r@cor:covering-questions@cref}{{[corollary][1][7]7.1}{}{}{}{}}
\@namedef{r@thm:diamond-HFCw}{{6.1}{}{}{}{}}
\@namedef{r@thm:diamond-HFCw@cref}{{[theorem][1][6]6.1}{}{}{}{}}
\@namedef{r@cor:further-structure}{{7.3}{}{}{}{}}
\@namedef{r@cor:further-structure@cref}{{[corollary][3][7]7.3}{}{}{}{}}
\let\RA@cref@hyperlink\cref@hyperlink
\def\cref@hyperlink#1#2#3\@nil{%
  \begingroup
  \edef\RA@target{#2}%
  \ifx\RA@target\@empty
    #3%
  \else
    \RA@cref@hyperlink{#1}{#2}#3\@nil
  \fi
  \endgroup}
\makeatother

\begin{theorem}[CH]\label{thm:CH-hL-power}
There is a zero-dimensional non-D space $X$ of cardinality $\aleph_1$ whose
countable power $X^\omega$ is hereditarily Lindel\"of.
\end{theorem}

\begin{theorem}\label{thm:d}
The following assertions are equivalent:
\begin{enumerate}
\item $\fd=\aleph_1$.
\item There is a regular Lindel\"of space of cardinality $\aleph_1$ that is not a D-space.
\item There is a regular hereditarily Lindel\"of space of cardinality $\aleph_1$ that is not a D-space.
\end{enumerate}
\end{theorem}

\begin{theorem}[$\fd=\aleph_1$]\label{thm:d-power}
There is a zero-dimensional, hereditarily Lindel\"of non-D space $X$ of
cardinality $\aleph_1$ such that $X^\omega\x M$ is Lindel\"of for every
second countable space $M$.
In particular, $X^\omega$ and all finite powers of $X$ are Lindel\"of.
\end{theorem}

\begin{theorem}[$\diamondsuit$]\label{thm:strong-HFCw}
There is a zero-dimensional, strongly $\HFCw$ space $X$ of cardinality $\aleph_1$ that is not D.
Its countable power $X^\omega$ is hereditarily Lindel\"of.
\end{theorem}

\section{Conclusion}
\label{sec:conclusion}\label{sec:classified-implications}

The preceding results give consistent negative answers to the following
questions about regular spaces and the D-property.
The hereditary Lindel\"of non-D examples already follow from
$\fd=\aleph_1$.
Hereditary Lindel\"ofness of the countable power is obtained under CH in
\cref{thm:CH-hL-power}. By \cref{cor:nonCH-hL-power}, such examples also exist
in models with $\fd=\aleph_1<\fc$.
The conclusions involving $\HFCw$ or countable closures use $\diamondsuit$.
The structural and game conclusions are proved in
\Cref{sec:open-problem-consequences}.

\Cref{thm:main} shows that regular (hereditarily) Lindel\"of spaces are not provably D
\begin{citelist}
\cite[Question~1]{Eisworth2007};
\cite[Problem~14]{HrusakMoore2007};
\cite[Problem~10]{SoukupSzeptycki2019}
\end{citelist}.
It also gives a zero-dimensional example of the kind sought by Soukup and Szeptycki~\cite[Question~5.5]{SoukupSzeptycki2012}.

By a theorem of Tall~\cite[Theorem~6(a)]{Tall1995}, our hereditarily Lindel\"of example is indestructibly Lindel\"of.
This gives a consistent negative answer to a question of Aurichi and Tall, whether every
regular indestructibly Lindel\"of space is D~\cite[Section~3, item~1]{AurichiTall2012}.

Every open cover of every subspace of our example has a clopen partition refinement.
As proved in \cref{cor:covering-questions}, these partitions give all the
listed covering properties, as well as normality and $aD$.
This settles the questions whether paracompactness,
subparacompactness, metacompactness, or screenability implies D~\cite[Questions~2, 3, and~6]{Eisworth2007},
and the corresponding questions for the additional covering properties listed by Gruenhage~\cite[Section~3, p.~16]{GruenhageSurvey2011}.
In particular, it answers the D-parts of problems of Arhangel'skii concerning countably metacompact weakly $\theta$-refinable spaces and screenable Tychonoff spaces
\begin{citelist}
\cite[Problems~1.18 and~1.22]{Arhangelskii2005};
\cite[Questions~4 and~6]{Eisworth2007}
\end{citelist},
and a weakly submetacompact alternative of Gruenhage~\cite[Question~3.1(1)]{GruenhageSurvey2011}.

Every countable subset of either diamond space in
\cref{thm:diamond-HFCw,thm:strong-HFCw} has countable closure
(\Cref{cor:further-structure}).
Thus every countable subset has Menger, indeed $\sigma$-compact, closure.
Since these spaces are $\HFCw$ and are not D, this answers both parts of a question of Soukup and Szeptycki~\cite[Problem~15]{SoukupSzeptycki2019},
and also  the $\sigma$-compact-closure variant asked immediately after that problem.
The space in \cref{thm:strong-HFCw} is, in addition, strongly $\HFCw$.

For the space $X$ in \cref{thm:CH-hL-power}, the countable power $X^\omega$ is hereditarily Lindel\"of.
Thus CH suffices to answer a question of Soukup and Szeptycki~\cite[Question~5.6]{SoukupSzeptycki2012}:
a regular space can have a hereditarily Lindel\"of countably infinite power without being D.
The fixed neighborhood assignment on our space $X$ gives SET a winning strategy in both games of Gruenhage and Szeptycki,
answering two of their questions~\cite[Questions~3 and~4]{GruenhageSzeptycki2011}.
The first question assumes hereditary Lindel\"ofness of $X^\omega$.
The second assumes that every finite power of every subspace of $X$ is Lindel\"of.
The example in \cref{thm:CH-hL-power} satisfies both hypotheses under CH;
the example in \cref{thm:strong-HFCw} is, in addition, strongly $\HFCw$.

The club principle $\clubsuit$, and even superclub, do not suffice
to obtain the existence assertion of \Cref{thm:main}, since they are consistent with
$\aleph_1<\fd$~\cite[Theorem~3]{GartiShelah2023}.

\end{document}